\documentclass{article}
\usepackage[utf8]{inputenc}
\usepackage[T1]{fontenc}    
\usepackage{hyperref}       
\usepackage{url}            
\usepackage{booktabs}       
\usepackage{amsfonts}       
\usepackage{nicefrac}       
\usepackage{microtype}      
\usepackage{lipsum}
\usepackage{amsmath}
\usepackage{amsthm}
\usepackage{amssymb}
\usepackage{mathtools}
\newtheorem{theorem}{Theorem}
\newtheorem{lemma}[theorem]{Lemma}

\theoremstyle{definition}
\newtheorem{defn}[theorem]{Definition}
\theoremstyle{definition}
\newtheorem{remark}[theorem]{Remark}
\numberwithin{theorem}{section}
\usepackage{graphicx}
\usepackage{subcaption}
\usepackage{natbib}
\usepackage{glossaries}

\usepackage{algorithm}
\usepackage[noend]{algpseudocode}
\makeatletter
\def\BState{\State\hskip-\ALG@thistlm}
\makeatother
\algdef{SE}[SUBALG]{Indent}{EndIndent}{}{\algorithmicend\ }%
\algtext*{Indent}
\algtext*{EndIndent}
\usepackage{authblk}

\title{Generic and Isometric Embeddings in Reservoir Computers}
\author{Allen G Hart}
\affil{University of Warwick, Coventry, CV4 7AL, UK}
\date{\today}

\begin{document}

\maketitle

\begin{abstract}
    We prove that a generic reservoir system admits a generalized synchronization that is a topological embedding of the input system's attractor. We also prove that for sufficiently high reservoir dimension (given by Nash's embedding theorem) there exists an isometric embedding generalized synchronization. The isometric embedding can be constructed explicitly when the reservoir system and source dynamics are linear. 
\end{abstract}

\section*{Lead Paragraph}

Reservoir computing is a machine learning paradigm that uses a high-dimensional dynamical system to model and predict the behavior of an underlying time-dependent source system that is only partially observed. A key feature of this framework is the \emph{generalized synchronization}, where the reservoir learns to mirror the evolution of the source system. In the special case that the generalized synchronization is an embedding, the synchronization can not only predict the future observations of the source system but also faithfully reconstructs the hidden topological structure of the underlying dynamics. In this work, we show that such embeddings occur generically in reservoir systems, providing rigorous justification for a phenomenon long observed in practice. Moreover, we establish conditions under which reservoirs can achieve \emph{isometric embeddings}, preserving not just the topology but also the geometric properties such as lengths and angles. These results demonstrate that reservoirs have the capacity to represent dynamical systems without distortion, opening the way to more reliable learning, analysis, and discovery of invariants in nonlinear time series.

\section{Introduction}

Reservoir computing is a machine learning paradigm designed to efficiently process temporal data using a high-dimensional dynamical system called the \emph{reservoir}, which transforms input signals into rich internal states. When a reservoir is driven by a sequence of observations from a dynamical system, the combined system can exhibit \emph{generalized synchronization} (GS) \cite{Hart2021, Hart2020, HartThesis2022, Hart2024, Grigoryeva2021, Grigoryeva2023, HaraKokubu2022, Hu2022, Kobayashi2021, Lu2017, Lu2018, Lymburn2019Reservoir, Nazerian2023, Wu2024, Verzelli2020, Chen2022, Ohkubo2024, Parlitz2024, Platt2021, Platt2022, Ruffini2018} where the reservoir state $\mathbf{x}_k$ becomes functionally dependent on the input system’s state $\mathbf{m}_k$ via a map $f$, i.e., $\mathbf{x}_k \approx f(\mathbf{m}_k)$.

In the special case where the GS map $f$ is an embedding, the topological properties of the underlying dynamical system are preserved and faithfully replicated in the reservoir \cite{Hart2020, Verzelli2020, Pathak2017a}. This enables not only the prediction of future observations, but also the recovery of geometric information about the original system such as the eigenvalues of fixed points, Lyapunov exponents, and homology groups \cite{Hart2020, Pathak2017a}. Empirically, when generalized synchronization occurs, it often looks like an embedding. So in this paper we will prove that, under suitable conditions (inspired by Whitney \cite{Whitney1936}) a generic reservoir map possesses a generalized synchronization $f$ which is an embedding. This result is closely related to the celebrated Takens embedding theorem \cite{Takens1981Detecting, Grigoryeva2023}. We also prove that for sufficiently high reservoir dimension (given by the bounds of Nash's embedding theorem) there exists an isometric embedding generalized synchronization. These embeddings additionally preserve the angles and lengths of the source dynamics.

\section{Generic Embeddings}

First, we will define a reservoir system.

\begin{defn}
    (Reservoir System) Let $M$ be a smooth $q$ dimensional manifold and $\phi \in \text{Diff}^1(M)$ a diffeomorphism on $M$. Let $\omega \in C^1(M,\mathbb{R})$ be a scalar observation function. Let $F \in C^1(\mathbb{R}^N \times \mathbb{R}, \mathbb{R}^N)$ denote the reservoir map. Then the reservoir system $\Phi : M \times \mathbb{R}^N \to M \times \mathbb{R}^N$ is defined
    \begin{align*}
        \Phi(m,x) = (\phi(m), F(x,\omega(m))).
    \end{align*}
\end{defn}

Under certain conditions, the reservoir system has a (not necessarily unique \cite{Grigoryeva2021},\cite{ceni2020echo}) associated generalised synchronisation.

\begin{defn}
    (Generalized Synchronization) Let $U \subset \mathbb{R}^N$ be an open set and $\Phi$ a reservoir system. The pair $(\Phi,U)$ is said to have the generalized synchronization $f \in C^1(M,\mathbb{R}^N)$ on the $\phi$-invariant compact set $\mathcal{A} \subset M$ if the graph
    \begin{align*}
        G = \{ (m,f(m)) \ | \ m \in \mathcal{A} \}
    \end{align*}
    is invariant under the evolution operator $\Phi$ i.e. $\Phi(G) = G$ and $\phi(\mathcal{A}) = \mathcal{A}$ and the graph $G$ is an attractor in the sense that there exists an $m_0 \in \mathcal{A}$ such that for any $x_0 \in U$
    \begin{align*}
        \Phi^k(m_0,x_0) \xrightarrow[k \to \infty]{} (\phi^k(m_0),f\phi^k(m_0)).
    \end{align*}
\end{defn}

\begin{defn}
    For a given open set $U \subset \mathbb{R}^N$ and reservoir system $\Phi : M \times \mathbb{R}^N \to M \times \mathbb{R}^N$ let $\mathcal{F} \subset C^1(\mathbb{R}^N \times \mathbb{R}, \mathbb{R})$ be the interior of the set of all reservoir maps $F$ such that $(\Phi,U)$ has a GS $f$.
\end{defn}

Next we introduce Whitney's 2nd embedding Theorem - a central result stating that if the reservoir dimension $N$ is greater than twice the manifold dimension $q$ then generic maps $f \in C^1(M,\mathbb{R}^N)$ are embedding. This lower bound on the dimension $N$ will be used to prove the major result of this section; Theorem \ref{thm::generic_F_embedding_f}.

\begin{theorem}[Whitney \cite{Whitney1936,Lee2013}]
    Let $M$ be a smooth manifold of dimension $q$. 
    If $N > 2q$, then a generic map in $C^1(M,\mathbb{R}^N)$ is an embedding. 
    Equivalently, the subset of embeddings $E \subset C^1(M,\mathbb{R}^N)$ is open and dense.
\end{theorem}

Now we are ready to state and prove the major theorem - that for generic reservoir maps that have a GS, the associated GS is an embedding. The idea behind the proof is to define a continuous open mapping $\Psi : F \to f$ from the reservoir map to the GS $f$ and use the genericity of embeddings guaranteed by Whitney's theorem to obtain the genericity of the maps $F$. We will assume that the source system $\phi$ has the conditions necessary for Takens' Theorem to hold.

\begin{theorem} 
\label{thm::generic_F_embedding_f}
Suppose that $\phi \in \mathrm{Diff}^2(M)$ has finitely many periodic orbits in $\mathcal{A}$. Suppose that for each periodic orbit of $m \in \mathcal{A}$ with period $n \in \mathbb{N}$, the derivative $T_m \phi^{-n}$ has $q$ distinct eigenvalues. Suppose that $N > 2q$. Then for a generic observation function $\omega \in C^1(M,\mathbb{R})$ and generic reservoir map $F \in \mathcal{F}$, the generalized synchronization
\(
f \in C^1(M,\mathbb{R}^N)
\)
is an embedding of $\mathcal{A} \subset M$.
\end{theorem}

\begin{proof}
Let 
\[
    S := \{ f \in C^1(M,\mathbb{R}^N) \mid f \text{ is a GS for some $F \in \mathcal{F}$ } \}.
\]
 The set of generalised synchronisations \(S\) is open and nonempty by Lemma \ref{lem::open_nonempty}. Furthermore, the set of embeddings \(E \subset C^1(M,\mathbb{R}^N)\) is an open and dense subset by Whitney's embedding Theorem \cite{Whitney1936,Lee2013}. Now it follows (by Lemma \ref{lem::topology}) that \(S \cap E\) is a dense open subset of \(S\). Now define the map $\Psi : \mathcal{F} \to S \subset C^1(M,\mathbb{R}^N)$ as the mapping of the reservoir map $F \in \mathcal{F}$ to its associated GS $f \in C^1(M,\mathbb{R}^N)$.
Now using that $\Psi$ is continuous and open (Lemma \ref{lem::open_and_cont}), and Lemma \ref{lem::topology2}, it follows that
$\Psi^{-1}(S \cap E)$
is a generic subset of \(\Psi^{-1}(S) = \mathcal{F} \), completing the proof.
\end{proof}

\begin{remark}
The periodic-orbit non-degeneracy assumption (distinct eigenvalues for each $T_m\phi^{-n}$) mirrors the standard hypotheses in Takens-type embedding theorems, ensuring transversality of delayed coordinates (see Huke’s formulation of Takens’ Theorem). The assumption is not very restrictive - and is in fact satisfied for generic evolution operators $\phi$. This was established independently by Kupka \cite{Kupka1963}\cite{Kupka1963} and \cite{Smale1963} Smale. In fact the original Takens embedding theorem \cite{Takens1981Detecting} did not spell out the periodic-orbit non-degeneracy assumptions in detail - and was shown by Takens to hold for a generic pair $(\phi,\omega)$.
\end{remark}

\begin{remark}
    Theorem \ref{thm::generic_F_embedding_f} is closely related to the result in \cite{Grigoryeva2023}, which shows almost all \emph{linear} reservoir systems reservoir produce a generalized synchronization that is generically an embedding of the attractor into the reservoir state space \cite{Grigoryeva2023}. We can view this result as a sort of special case of Theorem \ref{thm::generic_F_embedding_f} which holds for generally nonlinear systems.
\end{remark}

\begin{remark}
    The original Takens theorem \cite{Takens1981Detecting} is set on a compact manifold $M$, but the proof goes through using the same arguments when we are considering the restriction of the delay map to a compact attractor $\mathcal{A} \subset M$. Moreover, the dimension $N$ of the reservoir space $\mathbb{R}^N$ can be bounded below by some integer lower than $2q + 1$ if the attractor $\mathcal{A}$ has lower dimension than the manifold dimension. This sharper bound is outside the scope of this paper but is analysed in \cite{sauer1991embedology}.
\end{remark}

\section{Discussion}

Theorem~\ref{thm::generic_F_embedding_f} states that if the reservoir 
dimension $N$ is larger than twice the manifold dimension $q$ (following Whitney), then for 
a generic reservoir map $F \in \mathcal{F}$, the associated GS 
$f \in C^1(M,\mathbb{R}^N)$ is an embedding of the attractor 
$\mathcal{A} \subset M$. Hence, if we take a random sample from a 
non-singular distribution over $\mathcal{F}$ (or an open subset of 
$\mathcal{F}$), then we will almost surely obtain a reservoir map $F$ 
whose GS $f$ is an embedding.

In the reservoir computing paradigm, we typically generate a reservoir 
map $F$ by randomly drawing weights from a parameterized class. If this 
class has the universal approximation property—as is the case for 
Echo State Networks (ESNs)—then it has been observed empirically that, 
under suitable hyperparameter choices, almost all realizations of the 
random weights yield an embedding GS 
\cite{Haluszczynski2019,Gauthier2021,Wikner2021}. 
This empirical fact is explained precisely by 
Theorem~\ref{thm::generic_F_embedding_f}, together with the universal 
approximation property of ESNs 
\cite{GrigoryevaOrtega2018}.

The results in this section are set in discrete time, so it is natural to wonder if they hold in continuous time as well. The arguments in this paper directly use Takens' embedding theorem, which does not have a direct continuous-time analogue, so the problem is more complicated than simply repeating the same reasoning for vector fields. There is some analysis of embedding GS in continuous time in \cite{Hart2024}, \cite{wong2024contraction}, and it seems plausible that similar techniques can be used to prove the result.

\section{Isometric Embeddings}

Having established that, generically, an embedding GS exists, we now move on to consider under what conditions an isometric embedding GS exists.
An isometric embedding preserves both topological properties and geometrical properties, including lengths, angles, and curvatures.

\begin{defn}
    (Isometric embedding) Let $g$ be a Riemannian metric on the manifold $M$. A map $f \in C^1(M,\mathbb{R}^N)$ is an isometric embedding of $\mathcal{A} \subset M$ if it is an embedding of $\mathcal{A}$ and if
    \begin{align*}
        D_mf v \cdot D_mf w = g_m(v,w) \qquad \forall v,w \in T_m M, \quad \forall m \in \mathcal{A}.
    \end{align*}
\end{defn}

The celebrated Nash-embedding Theorem states that for large enough $N$ any manifold $M$ can be isometrically embedded into $\mathbb{R}^N$.

\begin{theorem}
\label{thm::nash_embedding}
    (Nash \cite{nash1954}\cite{delellis2016masterpieces}) Let $(M,g)$ be a Riemannian manifold of dimension $q$. Then for $N \geq 2q+1$
    there exists an isometric embedding $\iota \in C^1(M,\mathbb{R}^N)$.
\end{theorem}

We will use Theorem \ref{thm::nash_embedding} to prove that for sufficiently large reservoir dimension $N$ there exists a reservoir map $F$ that has an isometric embedding GS.

\begin{theorem}
\label{thm:exist_isometric_GS}
Suppose that $\phi \in \mathrm{Diff}^2(M)$ has finitely many periodic orbits in $\mathcal{A}$. Suppose that for each periodic orbit $m \in \mathcal{A}$ with period $n \in \mathbb{N}$, the derivative $T_m \phi^{-n}$ has $q$ distinct eigenvalues.
Let $q=\dim M$ and suppose $N$ satisfies the Nash bound of Theorem~\ref{thm::nash_embedding}; $N \geq 2q + 1$.
Then for generic observation function $\omega \in C^1(M,\mathbb{R}^N)$ there exists a reservoir map $F \in \mathcal{F}$ whose generalized synchronization 
$\iota \in C^1(M,\mathbb{R}^N)$ is an isometric embedding of $\mathcal{A}\subset M$.
\end{theorem}

\begin{proof}
By Theorem~\ref{thm::nash_embedding}, there exists an isometric embedding $\iota : M \to \mathbb{R}^N$.
By Theorem~\ref{thm::generic_F_embedding_f}, there exists a reservoir map $F\in\mathcal{F}$ with embedding GS $f : M\to \mathbb{R}^N$.
Let $h$ be a diffeomorphism defined on a neighborhood of $f(M) \subset \mathbb{R}^N$ such that 
$h\circ \iota = f$ on $M$.
Define the conjugated reservoir map
\[
F^*(x,z) := h^{-1}\big(F(h(x),z)\big).
\]
Then the reservoir system
\begin{align*}
    \Phi^*(m,x) = (\phi(m),F^*(x,\omega(m)))
\end{align*}
has the generalized synchronization $\iota \in C^1(M,\mathbb{R}^N)$ which is an isometric embedding.
\end{proof}

\begin{remark}
In the proof we constructed a reservoir map $F^*$ by conjugating $F$ with the (local) diffeomorphism $h$. This is reminiscent of system isomorphism \cite{Tsuda1992} \cite{GrigoryevaOrtega2021} \cite{Grigoryeva2023}.
\end{remark}

\begin{defn}[System isomorphism \cite{Tsuda1992} \cite{GrigoryevaOrtega2021} \cite{Grigoryeva2023}]
\label{def:system_isomorphism}
Two reservoir maps $F, F^* \in C^1(\mathbb{R}^N \times \mathbb{R},\mathbb{R}^N)$ are said to be 
\emph{isomorphic} if there exists a diffeomorphism $h \in \text{Diff}^1(\mathbb{R}^N)$ such that
\[
    F^*(x,z) = h \, F(h^{-1}(x),z).
\]
\end{defn}

\section{Linear Reservoir Systems}

In the special case of the linear reservoir system where
\begin{align*}
    F(x,z) = Ax + Cz
\end{align*}
for $A \in \mathbb{R}^{N\times N}$ and $C \in \mathbb{R}^N$, with $\phi \in \text{Diff}^1(M)$ linear, and $\omega \in C^1(M,\mathbb{R})$, we can construct reservoir systems with isometric embedding generalized synchronizations explicitly. This simultaneous linearization of both the evolution operator $\phi \in \text{Diff}^1(M)$ and reservoir system arises in the vicinity of hyperbolic fixed points \cite{Hartman1960} \cite{Grobman1959}. Furthermore, we will show that any linear reservoir system that has an embedding GS is isomorphic to one that is isometrically embedding.

\begin{defn}[Linear system isomorphism]
\label{def:linear_system_isomorphism}
Let $F(x,z) = Ax + Cz$ and $F^*(x,z) = A^*x + C^* z$ be linear reservoir maps with 
$A,A^* \in \mathbb{R}^{N \times N}$ and $C,C^* \in \mathbb{R}^N$. 
We say that $F$ and $F^*$ are \emph{linearly isomorphic} if there exists an invertible 
matrix $H \in GL(N)$ such that
\[
    A^* = H A H^{-1}, 
    \qquad 
    C^* = H C.
\]
Equivalently, $F$ and $F^*$ are isomorphic in the sense of 
Definition~\ref{def:system_isomorphism} with $h(x)=Hx$.
\end{defn}

\begin{theorem}
Let $\phi \in \text{Diff}^1(M)$ be a linear map, and $\{ (\lambda_j, v_j) \}$ represent the eigenvalue-eigenvector pairs of $T_m \phi^{-1}$. Suppose the eigenvalues of $T_m\phi^{-1}$ are distinct and that $D_m\omega v_i \neq 0$. Suppose that $A \in \mathbb{R}^{q\times q}$ and $C \in \mathbb{R}^q$ satisfy
\begin{align*}
    \sum_{k=0}^{\infty} A^k C < \infty.
\end{align*}
Furthermore, suppose that the vectors
\begin{align*}
    \{(\mathbb{I} - \lambda_j A)^{-1}C\}_{j = 1}^q
\end{align*}
are linearly independent.
Let $\{ w_i \}$ be an orthonormal basis with respect to the Riemannian metric $g$. We can express these vectors with respect to the eigenbasis $\{v_i\}$ as follows
\begin{align*}
    w_i = \sum_{j=1}^q v_j Q_{ji}
\end{align*}
for a matrix $Q$.
Let $P$ be a matrix with $j$th column $D_m\omega v_j (\mathbb{I} - \lambda_j A)^{-1}C$. Then for any rotation matrix $R$ the reservoir map
\begin{align*}
    F^*(x,z) = A^* x + C^* z
\end{align*}
where
\begin{align*}
    A^* &= (PQR)^{-1} A P Q R \\
    C^* &= (PQR)^{-1} C
\end{align*}
is linearly isomorphic to $F$ and has generalized synchronization
\begin{align*}
    f^* = \sum_{k=0}^{\infty} A^{*k} C^* \omega\phi^{-k}
\end{align*}
    which is an isometric embedding.
\end{theorem}

\begin{proof}
To establish isometry we need to show that 
\begin{align*}
    g(u,y) = D_m f^*u \cdot D_m f^* y.
\end{align*}
So we start with
    \begin{align*}
    D_m f = \sum_{k = 0}^{\infty} A^k C D_{m} \omega T_{m} \phi^{-k}
\end{align*}
then observe that
\begin{align*}
D_m f w_i = PQ e_i 
\end{align*}
for canonical unit vector $e_i$.
Now introduce an arbitrary rotation matrix $R$, which represents the arbitrary choice of orthonormal basis with respect to $g$. Then

\begin{align*}
    A^* &= (PQR)^{-1} A P Q R \\
    C^* &= (PQR)^{-1} C.
\end{align*}

Thus

\begin{align*}
D_m f^* w_i = R^{-1} e_i.
\end{align*}

Now let

\begin{align*}
    u = \sum_{i=1}^q \alpha_i w_i, \qquad y = \sum_{i=1}^q \beta_i w_i
\end{align*}

for scalars $\{ \alpha_i \}_{i=1}^q$ and $\{ \beta_i \}_{i=1}^q$ and observe that

\begin{align*}
    g_m(u,y) = \sum_{i=1}^q \alpha_i \beta_i = D_m f^* u \cdot D_m f^* y.
\end{align*}
\end{proof}

\begin{remark}
In the case $N>q$, the choice of $S \in GL(N)$ completing the columns of $PQ$ is not unique. 
One may choose any $S$ such that 
\[
S e_i = (PQ) e_i, \quad i=1,\dots,q,
\]
and then for any block rotation $\overline{R} \in O(N)$ of the form described in the proof,
\[
    A^* = (S\overline{R})^{-1} A \, S\overline{R}, 
    \qquad 
    C^* = (S\overline{R})^{-1} C.
\]
\end{remark}

\begin{remark}
The full-rank condition on the matrix \(P\) (ensuring \(\mathrm{rank}(P)=q\)) is not restrictive in practice. Indeed, Proposition 4.4 in \cite{Grigoryeva2023} shows that when \(A\) and \(C\) are drawn randomly from any non-degenerate continuous distribution, the columns \((I-\lambda_j A)^{-1} C\) (for distinct \(\{\lambda_j\}\)) are almost surely linearly independent. Hence, the condition we impose holds with probability 1 in typical random reservoir initializations.
\end{remark}

\begin{remark}
It is notable that the isometric form constructed above does not depend on the specific values of 
the eigenvalues $\{\lambda_j\}$ beyond the necessary convergence conditions. 
In particular, \emph{isometry is achievable for any admissible set of eigenvalues}, as long as the 
rank condition on $P$ is met. This highlights that the spectral data of $T_m\phi^{-1}$ do not 
obstruct the construction of an isometric embedding.
\end{remark}

\section{Conclusion}

Generalized synchronization (GS) is the fundamental mechanism that allows 
reservoir systems to learn and represent dynamical systems. 
When the GS is an embedding, the reservoir state space provides a faithful 
reconstruction of the geometry of the underlying attractor, making possible the 
recovery of dynamical invariants and topological features. 
Our first main result shows that such embedding GS maps occur generically, 
formalizing what has been widely observed empirically in reservoir computing. 

Going further, we established that isometric embeddings can also arise in 
reservoir systems. Isometric GS maps not only preserve topology, but also the 
metric structure of the source dynamics, including angles and lengths. 
In the linear case, we showed that every reservoir system with embedding GS is isomorphic to a system with isometric GS. The isometry preserves the eigenvalues of the reservoir matrix, proving that many different spectra are consistent with an isometric GS. This reveals a stronger structural 
property of reservoirs: they can, in principle, represent the geometry of 
dynamical systems without distortion. Future work could investigate continuous-time analogues of these results and explore how isometric GS maps might improve learning, stability, and invariant discovery in practical reservoir computing applications.

\section*{Appendix}

\begin{lemma}
\label{lem::open_nonempty}
     Suppose that $\phi \in \mathrm{Diff}^2(M)$ has finitely many periodic orbits in $\mathcal{A}$. Suppose that for each periodic orbit of $m \in \mathcal{A}$ with period $n \in \mathbb{N}$, the derivative $T_m \phi^{-n}$ has $q$ distinct eigenvalues. Then for generic observation functions $\omega \in C^1(M,\mathbb{R})$ the set of generalised synchronisations $S \subset C^1(M,\mathbb{R}^N)$ is open and non-empty
\end{lemma}

\begin{proof}
    We let
    \begin{align*}
        F(x,z) = Ax + Cz
    \end{align*}
    where
    \begin{align*}
        A = 
        \begin{pmatrix}
            0 & 0 & \ldots & 0 & 0 \\
            1 & 0 & \ldots & 0 & 0 \\
            0 & 1 & \ldots & 0 & 0 \\
            \vdots & \vdots & \ddots & \vdots & \vdots & \\
            0 & 0 & \ldots & 1 & 0
        \end{pmatrix},
        \qquad C = 
        \begin{pmatrix}
            1 \\
            0 \\
            \vdots \\
            0
        \end{pmatrix}.
    \end{align*}
    so that
    \begin{align*}
        f(m) = 
        \begin{pmatrix}
            \omega\phi^{N-1}(m) \\
            \omega\phi^{N-2}(m) \\
            \vdots \\
            \omega(m)
        \end{pmatrix}.
    \end{align*}
    This is the Takens delay map which is an embedding GS under the specified conditions on $\phi$ and for generic $\omega \in C^1(M,\mathbb{R}^N)$. This establishes non-emptiness. 
    Now for any small perturbation of $f$ denoted $f'$ define a reservoir map $F' \in C^1(\mathbb{R}^N \times \mathbb{R} , \mathbb{R}^N)$ that satisfies
    \begin{align*}
        F'(f'(m),\omega\phi(m)) = f'\phi(m) \qquad \forall m \in \mathcal{A}
    \end{align*}
    and 
    \begin{align*}
        DF'(f'(m),\omega\phi(m))v = \frac{1}{2} v \qquad \forall m \in \mathcal{A}
    \end{align*}
    for $v$ in the normal space. Then we can smoothly continue $F'(x,z)$ on the remaining $(x,z) \in \mathbb{R}^N \times \mathbb{R}$. Then $f' \in S$. This establishes openness. 
\end{proof}

\begin{lemma}
\label{lem::topology}
    If $S$ is open and nonempty and $E$ is dense and open (all in a topological space $X$), then
    $S \cap E$ is a dense open subset of $S$ (with the subspace topology).
\end{lemma}

\begin{proof}
\emph{Openness.} Since $S$ and $E$ are open in $X$, their intersection $S\cap E$ is open in $X$, hence open in $S$ with the subspace topology.

\emph{Density.} Using the subspace–closure identity
\[
\overline{S \cap A}^{\,S}= S \cap \overline{A}^{\,X} \qquad(\text{for any }A\subset X),
\]
which is Proposition 16.4 in \cite{Munkres2000Topology}, we have
\[
\overline{S\cap E}^{\,S} \;=\; S \cap \overline{E}^{\,X}.
\]
Because $E$ is dense in $X$, $\overline{E}^{\,X}=X$, so $\overline{S\cap E}^{\,S}=S$. Hence $S\cap E$ is dense in $S$.
\end{proof}

\begin{lemma}
    \label{lem::open_and_cont}
    $\Psi : \mathcal{F} \to S$ is open and continuous.
\end{lemma}

\begin{proof}
    First we define an equivalence relation $\sim$ on $\mathcal{F}$ as $F \sim F'$ if $F, F' \in \mathcal{F}$ have generalised synchronisations that are equal on $\mathcal{A} \subset M$.
    Now the quotient map $\pi : \mathcal{F} \to \mathcal{F} / \sim$ is open, continuous (in the quotient topology) and surjective. Now we define the map $\chi : (\mathcal{F} / \sim) \to S$ as mapping the equivalence class $[F] \in \mathcal{F} / \sim$ to the common GS $f \in S \subset C^1(M,\mathbb{R}^N)$. Now we observe that $\Psi = \chi \circ \pi$. Now using that $\chi : (\mathcal{F} / \sim) \to S$ is a homeomorphism by construction, is follows the composition $\Psi = \chi \circ \pi$ is open and continuous which completes the proof.
\end{proof}

\begin{lemma}
\label{lem::topology2}
Let $X,B$ be topological spaces and let $\Psi:X\to B$ be continuous and open.
If $A\subset B$ is comeager (generic), then $\Psi^{-1}(A)$ is comeager (generic) in $X$.
In particular, if $X$ is a Baire space, then $\Psi^{-1}(A)$ is dense in $X$.
\end{lemma}

\begin{proof}
Write $A=\bigcap_{n=1}^\infty U_n$ with each $U_n$ open and dense in $B$.
By continuity, $\Psi^{-1}(U_n)$ is open in $X$ for each $n$.
To see density, let $V\subset X$ be nonempty and open.
Since $\Psi$ is open, $\Psi(V)$ is open in $B$, hence $\Psi(V)\cap U_n\neq\varnothing$ (as $U_n$ is dense).
Pick $b\in \Psi(V)\cap U_n$; then there exists $x\in V$ with $\Psi(x)=b\in U_n$, so $x\in V\cap \Psi^{-1}(U_n)$.
Thus every nonempty open $V$ meets $\Psi^{-1}(U_n)$, i.e. $\Psi^{-1}(U_n)$ is dense in $X$.
Therefore $\Psi^{-1}(A)=\bigcap_{n=1}^\infty \Psi^{-1}(U_n)$ is comeager in $X$.
If $X$ is Baire, every comeager subset is dense, proving the last claim.
\end{proof}

\bibliographystyle{unsrt}
\bibliography{references}

\end{document}